\documentclass[reqno,11pt]{amsart}
\usepackage[margin=1in]{geometry}
\usepackage{amsthm, amsmath, amssymb, bm}
\usepackage[dvipsnames]{xcolor}
\usepackage{tikz}
\usepackage[utf8]{inputenc}
\usepackage[T1]{fontenc}
\usepackage[textsize=small,backgroundcolor=orange!20]{todonotes}
\usepackage[hidelinks]{hyperref}
\usepackage{url}
\usepackage[noabbrev,capitalize]{cleveref}
\crefname{equation}{}{}
\usepackage{xcolor,graphics}

\newtheorem{theorem}{Theorem}[section]
\newtheorem{proposition}[theorem]{Proposition}
\newtheorem{lemma}[theorem]{Lemma}

\newtheorem{conjecture}[theorem]{Conjecture}
\newtheorem*{question*}{Question} \Crefname{question}{Question}{Questions}

\theoremstyle{definition}

\newtheorem{question}[theorem]{Question}

\theoremstyle{remark}

\DeclareMathOperator{\stab}{stab}
\newcommand{\Mod}[1]{\ (\text{mod}\ #1)}

\newcommand\numberthis{\addtocounter{equation}{1}\tag{\theequation}}

\title{$C$-$(k, \ell)$-Sum-Free Sets}

\author[Zhang]{Rachel Zhang}
\address{Department of Mathematics, MIT, Cambridge, MA 02139, USA}
\email{rachelyz@mit.edu}

\date{September 2020}

\begin{document}

\maketitle

\begin{abstract}
    The Minkowski sum of two subsets $A$ and $B$ of a finite abelian group $G$ is defined as all pairwise sums of elements of $A$ and $B$: $A + B = \{ a + b : a \in A, b \in B \}$. The largest size of a $(k, \ell)$-sum-free set in $G$ has been of interest for many years and in the case $G = \mathbb{Z}/n\mathbb{Z}$ has recently been computed by Bajnok and Matzke. Motivated by sum-free sets of the torus, Kravitz introduces the \emph{noisy Minkowski sum} of two sets, which can be thought of as discrete evaluations of these continuous sumsets. That is, given a noise set $C$, the noisy Minkowski sum is defined as $A +_C B = A + B + C$. We give bounds on the maximum size of a $(k, \ell)$-sum-free subset of $\mathbb{Z}/n\mathbb{Z}$ under this new sum, for $C$ equal to an arithmetic progression with common difference relatively prime to $n$ and for any two element set $C$.

\end{abstract}

\section{Introduction}

Given a finite abelian group $G$ of order $n$, the \emph{Minkowski sum} of two subsets $A, B \subseteq G$ is the set of pairwise sums, defined to be
\[
    A + B = \{ a + b\ |\ a \in A, b \in B \}.
\]
We also use $A - B$ to denote the pairwise differences of elements of $A$ and $B$:
\[
    A - B = \{ a - b\ |\ a \in A, b \in B \}.
\]
Extending these definitions, for an integer $k \ge 1$, we define $kA = A + \dots + A$, where there are $k$ copies of $A$ in the summation. 

For integers $k, \ell \ge 1$, we say that $A$ is \emph{$(k, \ell)$-sum-free} if $kA \cap \ell A = \emptyset$. Let $\mu_{k, \ell}(G)$ denote the largest possible size of a $(k, \ell)$-sum-free subset, that is,
\[
    \mu_{k, \ell}(G) = \max \{ |A|\ |\ \text{$A \subseteq G$ is $(k, \ell)$-sum-free} \}.
\]
Note that if $k = \ell$, then $kA = \ell A$, so we may assume that $k > \ell$.

When $k = 2$ and $\ell = 1$, we see that $\mu_{2, 1}(\mathbb{Z}/n\mathbb{Z})$ is simply the maximal size of a normal sum-free set in $\mathbb{Z}/n\mathbb{Z}$. The value of $\mu_{2, 1}(\mathbb{Z}/n\mathbb{Z})$ was first calculated by Diamanda and Yap in 1969.

\begin{theorem}[\cite{diamanda-yap}, Lemma 3]
    For any positive integer $n$,
    \[
        \mu_{2, 1}(\mathbb{Z}/n\mathbb{Z}) = \max_{d | n} \left\{ \left\lceil \frac{d-1}{d} \right\rceil \cdot \frac{n}{d} \right\}.
    \]
\end{theorem}

In 2003, Hamidoune and Plagne extended this result to $(k, \ell)$-sum-free sets for which $n$ and $k - \ell$ are relatively prime:

\begin{theorem}[\cite{critical-pair}, Theorem 2.6]
    If $\gcd(n, k - \ell) = 1$, then
    \[
        \mu_{k, \ell}(\mathbb{Z}/n\mathbb{Z}) = \max \left\{ \left\lceil \frac{d-1}{k + \ell} \right\rceil \cdot \frac{n}{d} \right\}.
    \]
\end{theorem}

In 2018, Bajnok and Matzke finished the problem of determining $\mu_{k, \ell}(\mathbb{Z}/n\mathbb{Z})$ by calculating $\mu_{k, \ell}(\mathbb{Z}/n\mathbb{Z})$, even when $n$ and $k - \ell$ are not relatively prime.

\begin{theorem}[\cite{bajnok-max}, Theorem 6]
    For positive integers $n, k, \ell$ with $k > \ell$,
    \[
        \mu_{k, \ell}(\mathbb{Z}/n\mathbb{Z}) = \max_{d | n} \left\{ \left\lceil \frac{d - (\delta - r)}{k + \ell} \right\rceil \cdot \frac{n}{d} \right\},
    \]
    where $\delta = \gcd(n, k - \ell)$, $f = \left\lceil \frac{d - \delta}{k + \ell} \right\rceil$, and $r$ is the remainder of $\ell f$ modulo $\delta$. 
\end{theorem}

We'll be more interested in a slightly different version of the Minkowski sum, motivated by sum-free sets of the torus $\mathbb{T}$. In \cite{kravitz}, Kravitz suggests the following problem: Consider the map $\varphi: \mathcal{P}(\mathbb{Z}/n\mathbb{Z}) \rightarrow \mathcal{P}(\mathbb{T})$ from subsets of $\mathbb{Z}/n\mathbb{Z}$ to subsets of the torus defined by $\varphi(A) = \bigcup_{i \in A} \left( \frac{i}{n}, \frac{i + 1}{n} \right)$. Then, since $\left( \frac{i}{n}, \frac{i + 1}{n} \right) + \left( \frac{j}{n}, \frac{j + 1}{n} \right) = \left( \frac{i + j}{n}, \frac{i + j + 2}{n} \right)$, we have that $\varphi(A) + \varphi(B) = \varphi(A + B + \{ 0, 1 \})$. That is, the normal Minkowski sum of the union of certain intervals in the torus corresponds to a new kind of sum of subsets of $\mathbb{Z}/n\mathbb{Z}$.

With this motivation, Kravitz defines what we will call the \emph{noisy Minkowski sum} of two sets with a set $C$: given sets $A, B, C \subseteq G$, let 
\[
    A +_C B = A + B + C = \{ a + b + c\ |\ a \in A, b \in B, c \in C \}.
\]
For a given set $C$, this operation can be understood as taking the normal Minkowski sum and adding some noise given by $C$. Then, define $k *_C A = kA + (k-1)C$. Note that when $C = \{ 0 \}$, we recover the normal Minkowski sum. 

The quantity we are interested in is 
\[
    \mu_{k, \ell}^C(G) = \max \{ |A|\ |\ A \subseteq G,\ k *_C A \cap \ell *_C A = \emptyset \},
\]
the maximum size of a \emph{$C$-$(k, \ell)$-sum-free set} of $G$. 

In his paper, Kravitz asks about the value of $\mu_{k, \ell}^{\{ 0, 1 \}}(\mathbb{Z}/n\mathbb{Z})$. Note that maximal $\{ 0, 1 \}$-$(k, \ell)$-sum-free subsets of $\mathbb{Z}/n\mathbb{Z}$ correspond to maximal $(k, \ell)$-sum-free subsets of $\mathbb{T}$, restricting to sets of the form $\bigcup_{i \in I} \left( \frac{i}{n}, \frac{i + 1}{n} \right)$. Thus, calculating $\mu_{k, \ell}^{\{ 0, 1 \}}(\mathbb{Z}/n\mathbb{Z})$ answers a less granular version of the largest $(k, \ell)$-sum-free set in the torus.

In this paper, we address Kravitz's question about $\mu_{k, \ell}^{\{ 0, 1 \}}(\mathbb{Z}/n\mathbb{Z})$ as well as give bounds on $\mu_{k, \ell}^C(\mathbb{Z}/n\mathbb{Z})$ for some other values of $C$. In particular, for $C = \{ 0, 1, \dots, c-1 \}$, we prove the following bounds:

\begin{theorem} \label{thm:01c-main}
    For $c \ge 2$ and $C = \{ 0, 1, \dots, c-1 \}$, we have
    \[
        \left\lfloor \frac{n + 2(c-2) - r}{k + \ell} \right\rfloor - (c-2)
        \le \mu_{k, \ell}^C(\mathbb{Z}/n\mathbb{Z})
        \le \left\lfloor \frac{n + 2(c-2)}{k + \ell} \right\rfloor - (c-2),
    \]
    where $\delta = \gcd(n, k - \ell)$, and $r$ is the remainder of $-k \cdot \left\lfloor \frac{n + 2(c-2)}{k + \ell} \right\rfloor + (c-2)$ modulo $\delta$. 
\end{theorem}

Note that when $\gcd(n, k - \ell) = 1$, $r < \delta = 1$, so $r = 0$, and the two sides of the bound are equal. In this case (and in many others), by taking $c = 2$, this theorem gives an explicit answer to Kravitz's question. When the two sides are not equal, we have that $r < \delta = \gcd(n, k - \ell) \le k - \ell < k + \ell$, so the upper and lower bounds can differ by at most 1.

We also consider two element sets $C$ of the form $\{ 0, s \}$. We prove the following bounds on $\mu_{k, \ell}^{\{ 0, s \}}(\mathbb{Z}/n\mathbb{Z})$:

\begin{theorem}\label{thm:0s-main}
    For $k > \ell$, we have
    \begin{align*}
        \mu_{k, \ell}^{\{ 0, s \}}(\mathbb{Z}/n\mathbb{Z})
        &\ge \max \left\{ 
            \max_{e | d} \left\{ \mu_{k, \ell}(\mathbb{Z}/e\mathbb{Z}) \cdot \frac{n}{e} \right\},\ 
            \left\lfloor \frac{n + 2(s-1) - r}{k + \ell} \right\rfloor - (s - 1)
        \right\} \\
        \mu_{k, \ell}^{\{ 0, s \}}(\mathbb{Z}/n\mathbb{Z}) 
        &\le \max \left\{
            \max_{e | d} \left\{ \mu_{k, \ell}(\mathbb{Z}/e\mathbb{Z}) \cdot \frac{n}{e} \right\},\ 
            \left\lfloor \frac{n}{k + \ell} \right\rfloor
        \right\},
    \end{align*}
    where $\delta = \gcd(n, k - \ell)$ and $r$ is the remainder of $-k\cdot \left\lfloor \frac{n + 2(s-1)}{k + \ell} \right\rfloor +(s - 1)$ modulo $\delta$.
\end{theorem}

When $\max_{e | d} \left\{ \mu_{k, \ell}(\mathbb{Z}/e\mathbb{Z}) \cdot \frac{n}{e} \right\} \ge \lfloor \frac{n}{k + \ell} \rfloor$, we have that the two sides are in fact equal, i.e.
\[
    \mu_{k, \ell}^{\{ 0, s \}}(\mathbb{Z}/n\mathbb{Z})
    = \max_{e | d} \left\{ \mu_{k, \ell}(\mathbb{Z}/e\mathbb{Z}) \cdot \frac{n}{e} \right\}.
\]

Another specific case is when $s = p$ is prime and $p$ divides $k - \ell$. Then $\mu_{k, \ell}(\mathbb{Z}/p\mathbb{Z}) = 0$, so our inequality simplifies to
\[
    \left\lfloor \frac{n + 2(p-1) - r}{k + \ell} \right\rfloor - (p - 1)
    \le \mu_{k, \ell}^{\{ 0, p \}}(\mathbb{Z}/n\mathbb{Z}) 
    \le \left\lfloor \frac{n}{k + \ell} \right\rfloor.
\]

We will see in Section 2 that the bounds we've achieved for $C = \{ 0, 1, \dots, c-1 \}$ extend to any set that is an arithmetic progression of length $c$ with common difference relatively prime to $n$, and the bounds for $C = \{ 0, s \}$ hold for any two element set. In Section 3, we prove Theorem~\ref{thm:01c-main}, and in Section 4, we prove Theorem~\ref{thm:0s-main}. Finally in Section 5, we discuss some open question and conjectures.

\section{Transformations of $C$}

As there are many choices for $C$, we may seek to show that several sets are equivalent in the sense that they all have the same size of a maximal $C$-$(k, \ell)$-sum-free set. In this section we introduce two transformations of $C$ that preserve $\mu_{k, \ell}^C(\mathbb{Z}/n\mathbb{Z})$.

\begin{proposition} \label{prop:shift}
    For any $g \in \mathbb{Z}/n\mathbb{Z}$, we have that $\mu_{k, \ell}^C (\mathbb{Z}/n\mathbb{Z}) = \mu_{k, \ell}^{C + \{ g \}}(\mathbb{Z}/n\mathbb{Z})$.
\end{proposition}

\begin{proof}
    We have that 
    \begin{align*} 
        k *_C A 
        &= kA + (k-1)C \\
        &= kA + (k-1)(C + \{ g \}) + k \{ -g \} + \{ g \} \\
        &= k (A + \{ -g \}) + (k - 1)(C + \{ g \}) + \{ g \} \\
        &= k *_{C + \{ g \}} (A + \{ -g \}) + \{ g \},
    \end{align*}
    and similarly, $\ell *_C A = \ell *_{C + \{ g \}} (A + \{ -g \}) + \{ g \}$, therefore letting $B = A + \{ -g \}$, we have that 
    \[
        k *_C A = \ell *_C A 
        \Longleftrightarrow 
        k *_{C + \{ g \}} B = \ell *_{C + \{ g \}} B.
    \]
    Thus, the proposition follows. 
\end{proof}

\begin{proposition} \label{prop:mult}
    For any $g \in \mathbb{Z}/n\mathbb{Z}^\times$ and set $A$, let $A / g = \{ a / g : a \in A \}$. Then, $\mu_{k, \ell}^C (\mathbb{Z}/n\mathbb{Z}) = \mu_{k, \ell}^{C / g} (\mathbb{Z}/n\mathbb{Z})$.
\end{proposition}

\begin{proof}
    For any set $A \subseteq \mathbb{Z}/n\mathbb{Z}$, we have that $k *_C A \cap \ell *_C A = \emptyset$ iff $(k *_C A) / g \cap (\ell *_C A) / g = \emptyset$, or equivalently, $k *_{C / g} (A / g) \cap \ell *_{C / g} (A / g) = \emptyset$. 
\end{proof}

We will call the operation in Proposition~\ref{prop:shift} \emph{shift} and the operation in Proposition~\ref{prop:mult} \emph{multiplication}, as these are the respective operations the propositions allow us to perform on $C$. We say that two sets $C$ and $D$ are \emph{shift-mult-equivalent} if, by applying a sequence of shifts and multiplications to $C$, one can attain $D$. Note since shifts and multiplications are invertible and composable, such sequences define an equivalence relation. 

In fact, two sets $C$ and $D$ are shift-mult-equivalent iff $D = g(C + \{ h \})$ for some elements $h \in \mathbb{Z}/n\mathbb{Z}$ and $g \in \mathbb{Z}/n\mathbb{Z}^\times$. As a result, $\{ 0, 1, \dots, c-1 \}$ is shift-mult-equivalent to any length $c$ arithmetic progression with common difference relatively prime to $n$, and $\{ 0, c \}$ is shift-mult-equivalent to any two element set whose elements have difference $\Delta$ such that $\gcd(n, \Delta) = \gcd(n, c)$. 



\section{Maximal $\{ 0, 1, \dots, c-1 \}$-$(k, \ell)$-Sum-Free Sets}

In this section we look at $C = \{ 0, 1, \dots, c-1 \}$, with $c \ge 1$. Note that when $c = 0$, $C = \{ 0 \}$, which has already been thoroughly investigated. Hence, we consider $c \ge 2$.

When $c = 2$, we have that $C = \{ 0, 1 \}$. Note that shifting and multiplying gives us that any set of two elements whose difference is relatively prime to $n$ is shift-mult-equivalent to $C$. Using $(k, \ell)$-sum-free sets on the torus, we may give a natural upper bound for $\mu_{k, \ell}^{\{ 0, 1 \}}(\mathbb{Z}/n\mathbb{Z})$. In order to talk about sets of the torus $\mathbb{T}$, let $\mu^*$ denote the normalized Haar measure on $\mathbb{T}$ and let $\mu^*_{k, \ell}(\mathbb{T}) = \max \{ \mu^*(A)\ |\ A \subseteq \mathbb{T}, kA \cap \ell A = \emptyset \}$. Kravitz proved the following equality for maximal sum-free sets of $\mathbb{T}$:

\begin{theorem}[\cite{kravitz}, Theorem 1.3]
    For $k > \ell$, it holds that $\mu^*_{k, \ell}(\mathbb{T}) = \frac{1}{k + \ell}$.
\end{theorem}

Using the map $\phi : \mathcal{P}(\mathbb{Z}/n\mathbb{Z}) \rightarrow \mathcal{P}(\mathbb{T})$ we defined in the introduction, we can prove the following upper bound:

\begin{theorem} \label{thm:01easy}
    For $n, k, \ell \in \mathbb{N}$ with $k > \ell$, we have that $\mu_{k, \ell}^{\{ 0, 1 \}} \le \left\lfloor \frac{n}{k + \ell} \right\rfloor$.
\end{theorem}

\begin{proof}
    Recall that $\phi$ is defined on sets $A \subseteq \mathbb{Z}/n\mathbb{Z}$ by $\phi(A) = \bigcup_{i \in A} \left( \frac{i}{n}, \frac{i + 1}{n} \right)$ and that $\phi( A +_{\{ 0, 1 \}} B ) = \phi(A) + \phi(B)$. Therefore, $A$ is $\{ 0, 1 \}$-$(k, \ell)$-sum-free iff $\phi(A)$ is sum-free. Then,
    \[
        \mu_{k, \ell}^{\{ 0, 1 \}}(\mathbb{Z}/n\mathbb{Z}) 
        \le n \cdot \mu^*_{k, \ell}(\mathbb{T}) 
        = \frac{n}{k + \ell}.
    \]
    Since $\mu_{k, \ell}^{\{ 0, 1 \}}(\mathbb{Z}/n\mathbb{Z})$ is an integer, we in fact have that 
    \[
        \mu_{k, \ell}^{\{ 0, 1 \}}(\mathbb{Z}/n\mathbb{Z})
        \le \left\lfloor \frac{n}{k + \ell} \right\rfloor.
    \]
\end{proof}

This bound is sharp: for instance, for $n = 10$, $k = 2$, and $\ell = 1$, we have that $\{ 4, 5, 6 \}$ is a size $\left\lfloor \frac{n}{k + \ell} \right\rfloor = 3$ $\{ 0, 1 \}$-$(k, \ell)$-sum-free set.

In general, we may prove an upper bound on $\mu_{k, \ell}^C(\mathbb{Z}/n\mathbb{Z})$ with a different approach, which will align with Theorem~\ref{thm:01easy} in the case $C = \{ 0, 1 \}$. The upper bound we give is based on the following result of Kneser.

\begin{theorem}[Kneser \cite{kneser}]
    Let $G$ be a finite abelian group. For nonempty $A, B \subseteq G$ and $H = \stab(A + B)$ the stabilizer of $A + B$, then 
    \[
        |A + B| \ge |A + H| + |B + H| - |H|.
    \]
\end{theorem}

In order to apply this theorem, we need the following easy lemma.

\begin{lemma} \label{lemma:substab}
    For sets $A$ and $B$, $\stab(A) \subseteq \stab(A + B)$ as a subgroup inclusion.
\end{lemma}

\begin{proof}
    We have that 
    \[
        \stab(A) + (A + B)
        = (\stab(A) + A) + B 
        = A + B.
    \]
\end{proof}

Recall, by definition, that $A +_C B = A + B + C$, so this lemma also gives that $\stab(A) \subseteq \stab(A +_C B)$. 

Now, we can lower bound $|A +_C B|$.

\begin{lemma} \label{lemma:kneser_bound}
    When $C = \{ 0, 1, \dots, c-1 \}$ for $c \ge 2$, we have that
    \[ 
        |A +_C B| \ge \min \left\{ n, |A| + |B| + (c-2) \right\}.
    \]
\end{lemma}

\begin{proof}
    Let $K = \stab(A +_C B)$. Then, by Kneser's result, we have 
    \begin{align*}
        |A +_C B| 
        = |A + B + C| 
        &\ge |A + B + K| + |\{ 0, \dots, c-1 \} + K| - |K| \\
        &\ge |A + B| + |\{ 0, \dots, c-1 \} + K| - |K|. \numberthis \label{eqn:kneser-1}
    \end{align*}
    If $K = \mathbb{Z}/n\mathbb{Z}$, then this means that $A +_C B = \mathbb{Z}/n\mathbb{Z}$ since every element of $\mathbb{Z}/n\mathbb{Z}$ stabilizes $A +_C B$. Then, $|A +_C B| = n \ge \min \{ n, |A| + |B| + c - 2 \}$. So, assume $K \not= \mathbb{Z}/n\mathbb{Z}$. Note that if a subgroup $H$ stabilizes a set $X$, we must have that for any $x \in X$, $x + H \subseteq X$, so $X$ is a union of cosets of $H$. Then, if $[\mathbb{Z}/n\mathbb{Z} : K] \le c$, we have that $A +_C B$ is a union of cosets of $K$. However, if $a \in A$ and $b \in B$, then $\{ a + b, a + b - 1, \dots, a + b + c - 1 \} \subseteq A +_C B$, so at least one element of each coset of $K$ is in $A +_C B$. This implies that $A +_C B = \mathbb{Z}/n\mathbb{Z}$, which we've assumed is not true.
    
    Now, if $[\mathbb{Z}/n\mathbb{Z} : K] > c$, we have that $| \{ 0, \dots, c-1 \} + K| = c|K|$, so Equation~\ref{eqn:kneser-1} can be rewritten as
    \[
        |A +_C B| \ge |A + B| + (c-1)|K|.
    \]
    Now, let $H = \stab(A + B)$, so the above equation implies that
    \begin{align*}
        |A +_C B| 
        &\ge |A + H| + |B + H| + (c-1)|K| - |H| \\
        &\ge |A| + |B| + (c-1)|K| - |H|.
    \end{align*}
    By Lemma~\ref{lemma:substab}, $H$ is a subgroup of $K$. In particular, $|H| \le |K|$. Then, we can write
    \[
        |A +_C B| 
        \ge |A| + |B| + (c - 2)|H|
        \ge |A| + |B| + (c - 2),
    \]
    as desired.
\end{proof}

With this result, we can show an upper bound on $\mu_{k, \ell}^C(\mathbb{Z}/n\mathbb{Z})$.

\begin{theorem}
    We have that $\mu_{k, \ell}^C(\mathbb{Z}/n\mathbb{Z}) \le \left\lfloor \frac{n + 2(c-2))}{k + \ell} \right\rfloor - (c-2)$.
\end{theorem}

\begin{proof}
    Take $A$ to be a $C$-$(k, \ell)$-sum-free subset of $\mathbb{Z}/n\mathbb{Z}$. By iteratively applying Lemma~\ref{lemma:kneser_bound}, we find that
    \begin{align*} 
        |k *_C A| &\ge \min \{ n, k|A| + (k-1)(c-2) \} \\
        |\ell *_C A| &\ge \min \{ n, \ell|A| + (\ell-1)(c-2) \}.
    \end{align*} 
    Since both $k *_C A$ and $\ell *_C A$ are nonempty, we must have that $|k *_C A|, |\ell *_C A| < n$, so 
    \begin{align*} 
        |k *_C A| &\ge k|A| + (k-1)(c-2) \\
        |\ell *_C A| &\ge \ell|A| + (\ell-1)(c-2).
    \end{align*}
    Then since $k *_C A$ and $\ell *_C A$ are disjoint, 
    \begin{align*}
        n 
        \ge |k *_C A| + |\ell *_C A|
        \ge (k + \ell)|A| + (k + \ell - 2)(c - 2).
    \end{align*}
    Rearranging gives 
    \[
        |A| \le \frac{n + 2(c-2)}{k + \ell} - (c - 2).
    \]
    Since $|A|$ must be an integer, we must have
    \[
        |A| \le \left\lfloor \frac{n + 2(c-2)}{k + \ell} \right\rfloor - (c - 2).
    \]
\end{proof}

We denote this upper bound by $\chi(c, k, \ell) = \left\lfloor \frac{n + 2(c-2)}{k + \ell} \right\rfloor - (c-2)$. We will see in the lower bound that this upper bound is achieved or almost achieved by an interval. The analysis of the largest length of a $C$-$(k, \ell)$-sum-free interval will closely follow the methods used by Bajnok and Matzke in calculating the largest $(k, \ell)$-sum-free sets in \cite{bajnok-max}.

\begin{theorem}
    Let $\delta = \gcd(n, k - \ell)$ and $r$ denote the remainder of $-k\cdot \chi(c, k, \ell) - (k-1)(c-2)$ modulo $\delta$. Then,
    \[
        \mu_{k, \ell}^C(\mathbb{Z}/n\mathbb{Z}) \ge \left\lfloor \frac{n + 2(c-2) - r}{k + \ell} \right\rfloor - (c-2).
    \]
\end{theorem}

\begin{proof}
    We have that an interval $A = [a, \dots, a + m - 1]$ of length $m$ satisfies
    \begin{align*}
        k *_C A &= \{ ka, \dots, ka + km + (c-2)k - (c-1) \} \\
        \ell *_C A &= \{ \ell a, \dots, \ell a + \ell m + (c-2)\ell - (c-1) \},
    \end{align*}
    so 
    \[ 
        k *_C A - \ell *_C A = \{ (k-\ell)a - \ell m - (c-2)\ell + (c-1), \dots, (k-\ell)a + km + (c-2)k - (c-1) \}.
    \]
    We have that $A$ is $C$-$(k, \ell)$-sum-free iff $k *_C A - \ell *_C A$ does not contain $0$, so $A$ is $C$-$(k, \ell)$-sum-free iff there exists some $b \in \mathbb{Z}$ for which
    \begin{align*}
        bn + 1 &\le (k - \ell)a - \ell m - (c-2)\ell + (c-1) \\
        (b + 1)n - 1 &\ge (k - \ell)a + km + (c-2)k - (c-1),
    \end{align*}
    which can be rearranged to give
    \begin{align*}
        \ell m + (\ell-1)(c-2)
        &\le (k - \ell)a - bn 
        \le n - km - (k-1)(c-2) \\
        \Longleftrightarrow 
        \frac{\ell m + (\ell-1)(c-2)}{\delta}
        &\le \frac{k - \ell}{\delta} \cdot a - \frac{n}{\delta} \cdot b 
        \le \frac{n - km - (k-1)(c-2)}{\delta}.
    \end{align*}
    Since $\gcd \left( \frac{k - \ell}{\delta}, \frac{n}{\delta} \right) = 1$, any integer can be expressed as $\frac{k - \ell}{\delta} \cdot a - \frac{n}{\delta} \cdot b$ for some choice of $a$ and $b$. Thus, it suffices to show that there exists some integer $z$ for which 
    \[
        \frac{\ell m + (\ell-1)(c-2)}{\delta}
        \le z
        \le \frac{n - km - (k-1)(c-2)}{\delta},
    \]
    or equivalently,
    \[
        \frac{\ell m + (\ell-1)(c-2)}{\delta} 
        \le \left\lfloor \frac{n - km - (k-1)(c-2)}{\delta} \right\rfloor. \numberthis \label{eqn:interval_ineq}
    \]
    Since $\lfloor \frac{x}{\delta} \rfloor \ge \lfloor \frac{x - \delta + 1}{\delta} \rfloor$, we have that any $m$ that satisfies 
    \[
        \frac{\ell m + (\ell-1)(c-2)}{\delta} 
        \le \frac{n - km - (k-1)(c-2) - \delta + 1}{\delta}
    \] 
    must also satisfy \ref{eqn:interval_ineq}. But the above can be rewritten as 
    \[
        m \le \frac{n + 2c - \delta - 3}{k + \ell} - (c - 2),
    \]
    so there is an interval of size $\left\lfloor \frac{n + 2c - \delta - 3}{k + \ell} \right\rfloor - (c - 2) = \left\lfloor \frac{n + 2(c-2) - (\delta-1)}{k + \ell} \right\rfloor - (c - 2)$ that is $C$-$(k, \ell)$-sum-free. Since $\delta = \gcd(n, k - \ell) < k + \ell$, if we let $\gamma(n, k, \ell, c)$ denote the length of the longest $C$-$(k, \ell)$-sum-free interval, we have that $\gamma(n, k, \ell, c) \in \{ \chi(n, k, \ell, c), \chi(n, k, \ell, c) - 1 \}$.
    
    Since $r \equiv -k \cdot \chi(n, k, \ell, c) - (k-1)(c-2) \Mod{\delta}$ and $n \equiv 0 \Mod{\delta}$, we have that $\gamma(n, k, \ell, c) = \chi(n, k, \ell, c)$ iff
    \begin{align*}
        \ell\cdot \chi(n, k, \ell, c) + (\ell-1)(c-2) 
        &\le n - k\cdot \chi(n, k, \ell, c) - (k-1)(c-2) - r \\
        \Longleftrightarrow 
        \chi(n, k, \ell, c) 
        &\le \frac{n + 2(c-2) - r}{k + \ell} - (c-2).
    \end{align*}
    Since $\chi(n, k, \ell, c)$ is an integer, $\gamma(n, k, \ell, c) = \chi(n, k, \ell, c)$ exactly when 
    \[
        \chi(n, k, \ell, c)
        \le \left\lfloor \frac{n + 2(c-2) - r}{k + \ell} \right\rfloor - (c-2). \numberthis \label{eqn:bound_truth}
    \]
    Note that $\left\lfloor \frac{n + 2(c-2) - r}{k + \ell} \right\rfloor - (c-2)$ takes the value $\chi(n, k, \ell, c)$ exactly when Equation~\ref{eqn:bound_truth} holds and otherwise takes the value $\chi(n, k, \ell, c) - 1$, thus $\gamma(n, k, \ell, c) = \left\lfloor \frac{n + 2(c-2) - r}{k + \ell} \right\rfloor - (c-2)$ is the length of the longest $C$-$(k, \ell)$-sum-free interval in $\mathbb{Z}/n\mathbb{Z}$. In particular, this implies that 
    \[
        \mu_{k, \ell}^C(\mathbb{Z}/n\mathbb{Z})
        \ge \left\lfloor \frac{n + 2(c-2) - r}{k + \ell} \right\rfloor - (c-2).
    \]
\end{proof}

Combining the two bounds, and noting that $r$ is the remainder of $-k \cdot \chi(n, k, \ell, c) - (k-1)(c-2) = -k \cdot \left\lfloor \frac{n + 2(c-2)}{k + \ell} \right\rfloor + (c - 2)$ modulo $\delta$, we have proven Theorem~\ref{thm:01c-main}.

\section{$C$ of size $2$}

We now restrict our attention to $C$ of size $2$. First, by shifting, we can write $C$ as $\{ 0, s \}$. When $s$ is relatively prime to $n$, we can multiply $C$ by $s^{-1}$ so that $C = \{ 0, 1 \}$, which we have examined in the previous section. Therefore, we now consider sets $C = \{ 0, s \}$ for which $d = \gcd(s, n) \not= 1$.

We will prove the upper and lower bounds of Theorem~\ref{thm:0s-main} separately. Recall that $\gamma(n, k, \ell, s + 1) = \left\lfloor \frac{n + 2(s-1) - r}{k + \ell} \right\rfloor - (s - 1)$, so our lower bound will follow immediately from the following theorem:

\begin{theorem}
    For any $k > \ell$, 
    \[
        \mu_{k, \ell}^{\{ 0, s \}}(\mathbb{Z}/n\mathbb{Z}) 
        \ge \max \left\{ 
            \max_{e | d} \left\{ \mu_{k, \ell}(\mathbb{Z}/e\mathbb{Z}) \cdot \frac{n}{e} \right\},\ 
            \gamma(n, k, \ell, s + 1)
        \right\}.
    \]
\end{theorem}

\begin{proof}
    For any $e | d$, define $\psi_{n, e}$ to be the canonical projection from $\mathbb{Z}/n\mathbb{Z}$ onto $\mathbb{Z}/e\mathbb{Z}$. That is, $\psi_{n, e}(a)$ is equal to the remainder of $a$ modulo $e$. Suppose we have a maximal $(k, \ell)$-sum-free set $B$ in $\mathbb{Z}/e\mathbb{Z}$. Then, $A = \psi_{n, e}^{-1}(B)$ must also be $(k, \ell$)-sum-free and in fact $\{ 0, s \}$-$(k, \ell)$-sum-free. Since $A$ has size $\mu_{k, \ell}(\mathbb{Z}/e\mathbb{Z}) \cdot \frac{n}{e}$, by taking the maximum value over all $e | d$, we have that $\mu_{k, \ell}^{\{ 0, s \}}(\mathbb{Z}/n\mathbb{Z}) \ge \max_{e | d} \left\{ \mu_{k, \ell}(\mathbb{Z}/e\mathbb{Z}) \cdot \frac{n}{e} \right\}$.
    
    Because $\{ 0, s \} \subseteq \{ 0, 1, \dots, s \}$, any $\{ 0, 1, \dots, p \}$-$(k, \ell)$-sum-free set is also a $\{ 0, p \}$-$(k, \ell)$-sum-free set in $\mathbb{Z}/n\mathbb{Z}$. This gives the second lower bound, that $\mu_{k, \ell}^{\{ 0, s \}}(\mathbb{Z}/n\mathbb{Z}) \ge \gamma(n, k, \ell, s + 1)$.
\end{proof}

In order to handle the upper bound, we need the following lemma, which deals with the case that the stabilizer of $k *_{\{ 0, s \}} A$ does not contain $\langle d \rangle$. By $\langle x \rangle$, we mean the cyclic subgroup of $\mathbb{Z}/n\mathbb{Z}$ generated by $x$.

\begin{lemma} \label{lemma:0s-bound-stab}
    For a $\{ 0, s \}$-$(k, \ell)$-sum-free set $A$, let $K = \stab(k *_{\{ 0, s \}} A)$ and suppose that $K$ does not contain $\langle d \rangle$, the subgroup of $\mathbb{Z}/n\mathbb{Z}$ generated by $d$, (i.e. $K \not= \langle e \rangle$ for any $e | d$, $e > 1$). Then
    \[
        |A| \le \left\lfloor \frac{n}{k + \ell} \right\rfloor. 
    \]
\end{lemma}

\begin{proof}
    Our proof of this lemma follows from the following claim:
        For all $1 \le j \le k$, $|j *_{\{ 0, s \}} A| \ge j |A|$. 
        
    We prove this claim via induction. For $j = 1$, $|j *_{\{ 0, s \}} A| = |A|$, so the base case is true. Now for $2 \le j \le k$, suppose the claim is true for $j - 1$. Let $J = \stab(j *_{\{ 0, s \}} A)$. By Kneser, we have that 
    \[
        |j *_{\{ 0, s \}} A|
        \ge |A + ((j - 1) *_{\{ 0, s \}} A)| + |\{ 0, s \} + J| - |J|.
    \]
    By Lemma~\ref{lemma:substab}, $J$ is a subgroup of $K$, so $J$ also doesn't contain $\langle d \rangle$. Then, we have that $|\{ 0, s \} + J| = 2|J|$, so we can rewrite the above equation as 
    \begin{align*}
        |j *_{\{ 0, s \}} A|
        &\ge |A + ((j - 1) *_{\{ 0, s \}} A)| + |J| \\
        &\ge |A| + |(j - 1) *_{\{ 0, s \}} A| + |J| - |H| \\
        &\ge |A| + |(j - 1) *_{\{ 0, s \}} A| \\
        &\ge |A| + (j-1) |A| \\
        &= j |A|,
    \end{align*}
    where $H = \stab(A + ((j - 1) *_{\{ 0, s \}} A)) \subseteq J$ by Lemma~\ref{lemma:substab}, so $|H| \le |J|$. This completes the proof of the claim.
    
    To finish our proof, we note that $|k *_{\{ 0, s \}} A| \ge k |A|$ and $|\ell *_{\{ 0, s \}} A| \ge \ell |A|$. Since $k *_{\{ 0, s \}} A \cap \ell *_{\{ 0, s \}} A = \emptyset$, we have that $n \ge (k + \ell)|A| \implies |A| \le \lfloor \frac{n}{k + \ell} \rfloor$. 
\end{proof}

Now, we are ready to prove the upper bound.

\begin{theorem}
    For $k > \ell$, 
    \[
        \mu_{k, \ell}^{\{ 0, s \}}(\mathbb{Z}/n\mathbb{Z}) \\
        \le \max \left\{
            \max_{e | d} \left\{ \mu_{k, \ell}(\mathbb{Z}/e\mathbb{Z}) \cdot \frac{n}{e} \right\},\ 
            \left\lfloor \frac{n}{k + \ell} \right\rfloor
        \right\}.
    \]
\end{theorem}

\begin{proof}
    Suppose that $A$ is a $\{ 0, s \}$-$(k, \ell)$-sum-free set in $\mathbb{Z}/n\mathbb{Z}$. Once again define the map $\psi_{n, e}$ to be the canonical projection of $\mathbb{Z}/n\mathbb{Z}$ onto $\mathbb{Z}/e\mathbb{Z}$. 
    
    First, suppose that $K = \stab(k *_{\{ 0, s \}} A) = \langle e \rangle$ for some $e | d$, so $k *_{\{ 0, s \}} A$ is a union of cosets of $\langle e \rangle$ but is not equal to $\mathbb{Z}/n\mathbb{Z}$. If $|\psi_{n, e}(A)| > \mu_{k, \ell}(\mathbb{Z}/e\mathbb{Z})$, then since $\psi_{n, e}(k *_{\{ 0, s \}} A) = k \psi_{n, e}(A)$ and $k *_{\{ 0, s \}} A = \psi_{n, e}^{-1}(\psi_{n, e}(k *_{\{ 0, s \}} A)$ is a union of cosets, $k *_{\{ 0, s \}}$ and $\ell *_{\{ 0, s \}} A$ have nontrivial intersection, contradiction on $A$ being $\{ 0, s \}$-$(k, \ell)$-sum-free. Therefore, $|\psi_{n, e}(A) \le \mu_{k, \ell}(\mathbb{Z}/e\mathbb{Z})$. In particular, this gives the bound $|A| \le \mu_{k, \ell}(\mathbb{Z}/e\mathbb{Z}) \cdot \frac{n}{e}$. Taking the maximum value over all $e | d$ gives that $\mu_{k, \ell}^{\{ 0, s \}}(\mathbb{Z}/n\mathbb{Z}) \le \max_{e | d} \left\{ \mu_{k, \ell}(\mathbb{Z}/e\mathbb{Z} \cdot \frac{n}{e} \right\}$.
    
    Otherwise, $K$ does not contain $\langle d \rangle$. Then, by Lemma~\ref{lemma:0s-bound-stab}, $\mu_{k, \ell}^{\{ 0, s \}}(\mathbb{Z}/n\mathbb{Z}) \le \lfloor \frac{n}{k + \ell} \rfloor$. Combining the two cases gives the stated result.
\end{proof}

\section{Further Questions}

When $C = \{ 0, 1, \dots, c-1 \}$, the upper and lower bounds given for $\mu_{k, \ell}^C(\mathbb{Z}/n\mathbb{Z})$ often coincide. When they do not, the two bounds are $1$ apart, and we conjecture that the lower bound holds as equality.

\begin{conjecture}
    For $c \ge 2$ and $C = \{ 0, 1, \dots, c-1 \}$,
    \[
        \mu_{k, \ell}^{C}(\mathbb{Z}/n\mathbb{Z}) 
        = \max \left\{ 0,\ \left\lfloor \frac{n + 2(c-2) - r}{k + \ell} \right\rfloor - (c-2) \right\},
    \]
    where $\delta = \gcd(n, k - \ell)$, $\chi(c, k, \ell) = \left\lfloor \frac{n + 2(c-2)}{k + \ell} \right\rfloor - (c-2)$ and $r$ is the remainder of $-k\cdot \chi(c, k, \ell) - (k-1)(c-2)$ modulo $\delta$. That is, the largest $C$-$(k, \ell)$-sum-free subset of $\mathbb{Z}/n\mathbb{Z}$ is achieved by an interval.
\end{conjecture}

Notice that for small enough values of $n$, it is possible for $\left\lfloor \frac{n + 2(c-2) - r}{k + \ell} \right\rfloor - (c-2)$ to be negative, hence we need to compare it to $0$. 

Consider, for instance, $n = 40$, $k = 9$, and $\ell = 4$. For $C = \{ 0, 1 \}$, the upper bound in Theorem~\ref{thm:01c-main} is $3$, while the lower bound is $2$. A simple computer program verifies that the largest $\{ 0, 1 \}$-$(9, 4)$-sum-free set of $\mathbb{Z}/40\mathbb{Z})$ is $2$. For $C = \{ 0, 1, 2 \}$, the upper and lower bounds once again differ, being $2$ and $1$ respectively, and a computer program verifies that the maximum $\{ 0, 1, 2 \}$-$(9, 4)$-sum-free subset of $\mathbb{Z}/40\mathbb{Z}$ has length $1$.

We have checked using a computer that for all $2 \le c \le 10$, there are no values of $n$, $k$, and $\ell$ with $\ell < 10$, $k < 20$, and $n < 5(k + \ell)$ for which the upper and lower bounds of Theorem~\ref{thm:01c-main} differ and $\mu_{k, \ell}^{\{ 0, 1 \}}$ is equal to the upper bound, unless the upper bound is equal to $0$. 

For $C = \{ 0, s \}$, our bounds often are wider. We ask for a precise value of $\mu_{k, \ell}^C(\mathbb{Z}/n\mathbb{Z})$.

\begin{question}
    For $\gcd(n, s) \not= 1$, what is the value of $\mu_{k, \ell}^{\{ 0, s \}}(\mathbb{Z}/n\mathbb{Z})$?
\end{question}

Other values of $C$ may be interesting for study. Note that if $C$ is shift-mult-equivalent to a set contained in $\{ 0, 1, \dots, x \}$, then $A$ is $C$-$(k, \ell)$-sum-free if it is $\{ 0, 1, \dots, x \}$-$(k, \ell)$-sum-free, so the lower bound of Theorem~\ref{thm:01c-main} gives a lower bound on $\mu_{k, \ell}^C(\mathbb{Z}/n\mathbb{Z})$. An upper bound in some cases, with some careful considerations, can be attained following the methods of this paper, or by noting that if $\{ 0, x \}$ is a subset of $C$, then any $C$-$(k, \ell)$-sum-free set must also be $\{ 0, x \}$-$(k, \ell)$-sum-free, from which we attain an upper bound of $\left\lfloor \frac{n}{k + \ell} \right\rfloor$ by Theorem~\ref{thm:0s-main}. However, it is not known how to attain an upper bound better than $\left\lfloor \frac{n}{k + \ell} \right\rfloor$ in all cases.

\begin{question}
    What can we say about the value of $\mu_{k, \ell}^C(\mathbb{Z}/n\mathbb{Z})$ for other values of $C$?
\end{question}

\section{Acknowledgments}

This research was conducted at the University of Minnesota Duluth REU and was supported by NSF / DMS grant 1659047 and NSA grant H98230-18-1-0010.
I would like to thank Joe Gallian, who suggested the problem, and Mehtaab Sawhney and Aaron Berger, who gave comments on this paper.


\end{document}